\documentclass[12pt]{article}

\usepackage[margin=1in]{geometry}
\usepackage{enumerate}
\usepackage{amssymb}
\usepackage{mathrsfs}
\usepackage{graphicx}
\usepackage{amsthm}
\usepackage{amsmath}
\usepackage{bbm}

\newtheorem{definition}{Definition}[section]
\newtheorem{theorem}[definition]{Theorem}
\newtheorem{lemma}[definition]{Lemma}

\newtheorem{remark}[definition]{Remark}
\newtheorem{example}[definition]{Example}

\newtheorem{proposition}[definition]{Proposition}

\def\makeCal#1{%
\expandafter\newcommand\csname c#1\endcsname{\mathcal{#1}}}
\def\makeBf#1{%
\expandafter\newcommand\csname b#1\endcsname{\mathbf{#1}}}
\def\makeBb#1{%
\expandafter\newcommand\csname m#1\endcsname{\mathbb{#1}}}
\def\makeFrak#1{%
\expandafter\newcommand\csname f#1\endcsname{\mathfrak{#1}}}
\def\makeScr#1{%
\expandafter\newcommand\csname s#1\endcsname{\mathscr{#1}}}

\makeatletter
\count@=0
\loop
\advance\count@ 1
\edef\y{\@Alph\count@}%
\expandafter\makeCal\y
\expandafter\makeBf\y
\expandafter\makeBb\y
\expandafter\makeFrak\y
\expandafter\makeScr\y
\ifnum\count@<26
\repeat
\makeatother

\DeclareMathOperator{\Cat}{Cat}
\DeclareMathOperator{\m1}{\mathbbm{1}}

\usepackage{xypic}
\renewcommand{\cD}{\mathcal{D}}

\title{A generating function associated with the alternating elements in the positive part of $U_q(\widehat{\mathfrak{sl}}_2)$}
\author{Chenwei Ruan}
\date{}

\begin{document}
\maketitle

\begin{abstract}
The positive part $U_q^+$ of $U_q(\widehat{\mathfrak{sl}}_2)$ admits an embedding into a $q$-shuffle algebra. This embedding was introduced by M. Rosso in 1995. In 2019, Terwilliger introduced the alternating elements $\{W_{-n}\}_{n \in \mathbb{N}}$, $\{W_{n+1}\}_{n \in \mathbb{N}}$, $\{G_{n+1}\}_{n \in \mathbb{N}}$, $\{\tilde{G}_{n+1}\}_{n \in \mathbb{N}}$ in $U_q^+$ using the Rosso embedding. He showed that the alternating elements $\{W_{-n}\}_{n \in \mathbb{N}}$, $\{W_{n+1}\}_{n \in \mathbb{N}}$, $\{\tilde{G}_{n+1}\}_{n \in \mathbb{N}}$ form a PBW basis for $U_q^+$, and he expressed $\{G_{n+1}\}_{n \in \mathbb{N}}$ in this alternating PBW basis. In his calculation, Terwilliger used some elements $\{D_n\}_{n \in \mathbb{N}}$ with the following property: the generating function $D(t)=\sum_{n \in \mathbb{N}}D_nt^n$ is the multiplicative inverse of the generating function $\tilde{G}(t)=\sum_{n \in \mathbb{N}}\tilde{G}_nt^n$ where $\tilde{G}_0=1$. Terwilliger defined $\{D_n\}_{n \in \mathbb{N}}$ recursively; in this paper, we will express $\{D_n\}_{n \in \mathbb{N}}$ in closed form. 

\bigskip
\noindent {\bf Keywords}. Catalan word; generating function; $q$-shuffle algebra. \\
{\bf 2020 Mathematics Subject Classification}. Primary: 17B37. Secondary: 05E14, 81R50. 
\end{abstract}

\section{Introduction}
The quantized enveloping algebra $U_q(\widehat{\mathfrak{sl}}_2)$ has a subalgebra $U_q^+$, called the positive part \cite{CP,lusztig}. Both $U_q(\widehat{\mathfrak{sl}}_2)$ and $U_q^+$ appear in algebra \cite{baseilhac1,beck,DF,drinfeld,jimbo}, combinatorics \cite{ariki,FM,drg,TD00,ter_alternating,ter_catalan}, mathematical physics \cite{CP2,DFJMN,JM}, and representation theory \cite{CP,FR,XXZ01}. 

\medskip
\noindent In \cite{rosso}, M. Rosso introduced an embedding of the algebra $U_q^+$ into a $q$-shuffle algebra. 

\medskip
\noindent In \cite{damiani}, I. Damiani obtained a Poincar\'e-Birkhoff-Witt (or PBW) basis for $U_q^+$. In her construction the PBW basis elements $\{E_{n\delta+\alpha_0}\}_{n \in \mN}$, $\{E_{n\delta+\alpha_1}\}_{n \in \mN}$, $\{E_{(n+1)\delta}\}_{n \in \mN}$ are defined recursively. In \cite{ter_catalan}, P. Terwilliger expressed the Damiani PBW basis elements in closed form, using the Rosso embedding of $U_q^+$. 

\medskip
\noindent In \cite{ter_alternating}, Terwilliger used the Rosso embedding to obtain a type of element in $U_q^+$, said to be alternating. The alternating elements fall into four families, denoted by $\{W_{-n}\}_{n \in \mN}$, $\{W_{n+1}\}_{n \in \mN}$, $\{G_{n+1}\}_{n \in \mN}$, $\{\tilde{G}_{n+1}\}_{n \in \mN}$. It was shown in \cite{ter_alternating} that the alternating elements $\{W_{-n}\}_{n \in \mN}$, $\{W_{n+1}\}_{n \in \mN}$, $\{\tilde{G}_{n+1}\}_{n \in \mN}$ form a PBW basis for $U_q^+$; this PBW basis is called alternating. The alternating PBW basis was used in \cite{baseilhac1,baseilhac2,ter_compact,ter_central2,ter_central,ter_beck}. 

\medskip
\noindent In \cite[Theorem 9.15]{ter_alternating}, Terwilliger expressed $\{G_{n+1}\}_{n \in \mN}$ in terms of the alternating PBW basis. His answer involved the elements $\{D_n\}_{n \in \mN}$ with the following property: the generating function $D(t)=\sum_{n \in \mN}D_nt^n$ is the multiplicative inverse of the generating function $\tilde{G}(t)=\sum_{n \in \mN}\tilde{G}_nt^n$ where $\tilde{G}_0=1$.  In \cite[Section 11]{ter_alternating}, Terwilliger used $\{D_n\}_{n \in \mN}$ to describe how the Damiani PBW basis is related to the alternating PBW basis. 

\medskip
\noindent To motivate our results, we make some comments about $D(t)$. We mentioned that $D(t)$ is the multiplicative inverse of $\tilde{G}(t)$. Using this relationship, the elements $\{D_n\}_{n \in \mN}$ can be computed recursively from $\{\tilde{G}_n\}_{n \in \mN}$. Calculation of $D_n$ for $n \leq 3$ suggests that the elements $\{D_n\}_{n \in \mN}$ admit a closed form. Our goal in this paper is to express $\{D_n\}_{n \in \mN}$ in closed form. We will state our main result shortly. 

\medskip
\noindent First, we establish some conventions and notation. 

\medskip
\noindent In this paper, $\mN=\{0,1,2,\ldots\}$ is the set of natural numbers, and $\mZ=\{0,\pm 1,\pm 2,\ldots\}$ is the set of integers. The letters $n,k,i,j,r,s$ always represent an integer. Let $\mF$ denote a field. All algebras discussed are over $\mF$, associative, and with a multiplicative identity. Let $q$ denote a nonzero scalar in $\mF$ that is not a root of unity. For $n \in \mN$, define 
\[
[n]_q=\frac{q^n-q^{-n}}{q-q^{-1}}, \hspace{4em} [n]_q^!=[n]_q[n-1]_q \cdots [1]_q.  
\]
We interpret $[0]_q^!=1$. 

\medskip
\noindent We will be looking at the positive part of $U_q(\widehat{\mathfrak{sl}}_2)$, denoted by $U_q^+$ \cite{CP,lusztig}. The algebra $U_q^+$ is defined by generators $A,B$ and the $q$-Serre relations 
\begin{equation*}\label{qserre1}
A^3B-[3]_qA^2BA+[3]_qABA^2-BA^3=0, 
\end{equation*}
\begin{equation*}\label{qserre2}
B^3A-[3]_qB^2AB+[3]_qBAB^2-AB^3=0. 
\end{equation*}
Next we recall the Rosso embedding of $U_q^+$ into a $q$-shuffle algebra \cite{rosso}. Let $x,y$ denote noncommuting indeterminates (called \textit{letters}). Let $\mV$ denote the free algebra generated by $x,y$. A product $v_1v_2 \cdots v_n$ of letters is called a \textit{word}, and $n$ is called the \textit{length} of this word. The word of length $0$ is called \textit{trivial} and denoted by $\m1$. The words form a basis for the vector space $\mV$, called the \textit{standard basis}. The vector space $\mV$ admits another algebra structure called the $q$-shuffle algebra. The $q$-shuffle algebra was first introduced by Rosso \cite{rosso1,rosso} and later reinterpreted by Green \cite{green}. The $q$-shuffle product, denoted by $\star$, is defined recursively as follows: 
\begin{itemize}
\item For $v \in \mV$, 
	\begin{equation*}\label{star1}
	\m1 \star v=v \star \m1=v. 
	\end{equation*}
\item For the letters $u,v$, 
	\begin{equation*}\label{star2}
	u \star v=uv+vuq^{\langle u,v \rangle}, 
	\end{equation*}	
	where 
	 \[
	\langle x,x \rangle=\langle y,y \rangle =2, \hspace{4em}\langle x,y \rangle=\langle y,x \rangle=-2.
	\]
\item For a letter $u$ and a nontrivial word $v=v_1v_2 \cdots v_n$ in $\mV$, 
	\begin{equation*}\label{star3.1}
	u \star v=\sum_{i=0}^n v_1 \cdots v_iuv_{i+1} \cdots v_n q^{\langle u,v_1 \rangle+\cdots+\langle u,v_i \rangle}, 
	\end{equation*}
	\begin{equation*}\label{star3.2}
	v \star u=\sum_{i=0}^n v_1 \cdots v_iuv_{i+1} \cdots v_n q^{\langle u,v_n \rangle+\cdots+\langle u,v_{i+1} \rangle}. 
	\end{equation*}
\item For nontrivial words $u=u_1u_2 \cdots u_r$ and $v=v_1v_2 \cdots v_s$ in $\mV$, 
	\begin{equation*}\label{star4.1}
	u \star v=u_1((u_2 \cdots u_r) \star v)+v_1(u \star (v_2 \cdots v_s))q^{\langle v_1,u_1 \rangle+\cdots+\langle v_1,u_r \rangle}, 
	\end{equation*}
	\begin{equation*}\label{star4.2}
	u \star v=(u \star (v_1 \cdots v_{s-1}))v_s+((u_1 \cdots u_{r-1}) \star v)u_rq^{\langle u_r,v_1 \rangle+\cdots+\langle u_r,v_s \rangle}. 
	\end{equation*}
\end{itemize}
Note that the $q$-shuffle product of two words of length $l_1,l_2$ is a linear combination of words of length $l_1+l_2$. 

\medskip
\noindent Green showed in \cite{green} that $x,y$ satisfy the $q$-Serre relations in the $q$-shuffle algebra $\mV$: 
\begin{equation*}\label{starqserre1}
x \star x \star x \star y-[3]_qx \star x \star y \star x+[3]_qx \star y \star x \star x-y \star x \star x \star x=0, 
\end{equation*}
\begin{equation*}\label{starqserre2}
y \star y \star y \star x-[3]_qy \star y \star x \star y+[3]_qy \star x \star y \star y-x \star y \star y \star y=0. 
\end{equation*}
As a result there exists an algebra homomorphism $\natural$ from $U_q^+$ to the $q$-shuffle algebra $\mV$ that sends $A \mapsto x,B \mapsto y$. The map $\natural$ is injective by \cite[Theorem 15]{rosso}. Let $U$ denote the subalgebra of the $q$-shuffle algebra $\mV$ generated by $x,y$. By construction, the image of $\natural$ is $U$. 

\medskip
\noindent We now mention some special words in $\mV$ that will be useful later. 

\begin{definition}\label{def:alternating}\rm
(See \cite[Definition 5.2, Lemma 5.4]{ter_alternating}.) We define $G_0=\tilde{G}_0=\m1$. 

\medskip
\noindent For $n \in \mN$, define
\[G_{n+1}=G_n yx,\hspace{2em} \tilde{G}_{n+1}=\tilde{G}_n xy,\]
\[W_{-n}=\tilde{G}_nx,\hspace{2em}  W_{n+1}=y\tilde{G}_n.\] 

\medskip
\noindent The words $\{W_{-n}\}_{n \in \mN}$, $\{W_{n+1}\}_{n \in \mN}$, $\{G_{n}\}_{n \in \mN}$, $\{\tilde{G}_{n}\}_{n \in \mN}$ are called \textit{alternating}. 
\end{definition}

\begin{example}\label{ex:alternating}\rm
We have 
\begin{equation*}
\begin{split}
&W_0=x, \hspace{2em} W_{-1}=xyx, \hspace{2em} W_{-2}=xyxyx, \hspace{2em} \ldots \\
&W_1=y, \hspace{2em} W_2=yxy, \hspace{2em} W_3=yxyxy, \hspace{2em} \ldots \\
&G_1=yx, \hspace{2em} G_2=yxyx, \hspace{2em} G_3=yxyxyx, \hspace{2em} \ldots \\
&\tilde{G}_1=xy, \hspace{2em} \tilde{G}_2=xyxy, \hspace{2em} \tilde{G}_3=xyxyxy, \hspace{2em} \ldots
\end{split}
\end{equation*}
\end{example}

\noindent By \cite[Theorem 8.3]{ter_alternating}, the alternating words are contained in $U$. 

\medskip
\noindent It is shown in \cite[Proposition 5.10]{ter_alternating} that with respect to $\star$, $\{W_{-n}\}_{n \in \mN}$ mutually commute, $\{W_{n+1}\}_{n \in \mN}$ mutually commute, $\{G_{n}\}_{n \in \mN}$ mutually commute, and $\{\tilde{G}_{n}\}_{n \in \mN}$ mutually commute. Furthermore, by \cite[Theorem 10.1]{ter_alternating} the alternating words $\{W_{-n}\}_{n \in \mN}$, $\{W_{n+1}\}_{n \in \mN}$, $\{\tilde{G}_{n+1}\}_{n \in \mN}$ form a PBW basis for $U$. 

\medskip
\noindent In this paper, we focus on the alternating words $\{\tilde{G}_n\}_{n \in \mN}$. Consider their generating function 
\[\tilde{G}(t)=\sum_{n \in \mN} \tilde{G}_nt^n.\]
We will be discussing the multiplicative inverse of $\tilde{G}(t)$ with respect to $\star$. We now introduce this inverse. 

\begin{definition}\label{def:Dt}\rm
(See \cite[Definition 9.5]{ter_alternating}.) We define the elements $\{D_n\}_{n \in \mN}$ of $U$ in the following recursive way: 
\begin{equation}\label{convolution}
D_0=\m1, \hspace{4em} D_n=-\sum_{k=0}^{n-1}D_k \star \tilde{G}_{n-k} \hspace{1.5em} (n \geq 1). 
\end{equation}
Define the generating function 
\begin{equation*}\label{eq:Dt}
D(t)=\sum_{n \in \mN} D_nt^n.
\end{equation*}
\end{definition}

\begin{lemma}\label{lem:Dtinverse}\rm
(See \cite[Lemma 4.1]{ter_conjecture}.) The generating function $D(t)$ is the multiplicative inverse of $\tilde{G}(t)$ with respect to $\star$. In other words, 
\begin{equation}\label{eq:Dtinverse}
\tilde{G}(t) \star D(t)=\m1=D(t) \star \tilde{G}(t).
\end{equation}
\end{lemma}
\begin{proof}
The relation \eqref{eq:Dtinverse} can be checked routinely using \eqref{convolution}. 
\end{proof}

\noindent For $n \in \mN$ we can calculate $D_n$ recursively using \eqref{convolution}. 

\begin{example}\label{ex:Dn}\rm
We list $D_n$ for $0 \leq n \leq 3$. 
\[
D_0=\m1, \hspace{4em} D_1=-xy, \hspace{4em} D_2=xyxy+[2]_q^2xxyy, 
\]
\[
D_3=-xyxyxy-[2]_q^2xxyyxy-[2]_q^2xyxxyy-[2]_q^4xxyxyy-[2]_q^2[3]_q^2xxxyyy. 
\]
\end{example}

\noindent We are going to obtain a closed formula for $D_n$. 

\medskip
\noindent To motivate the formula, let us examine Example \ref{ex:Dn}. We can see that each $D_n$ is a linear combination of words of length $2n$, and each coefficient is equal to $(-1)^n$ times a square. Furthermore, the words appearing in the linear combination have a certain type said to be Catalan. We now recall the definition of a Catalan word. 

\begin{definition}\label{def:Cat}\rm
(See \cite[Definition 1.3]{ter_catalan}.) Define $\overline{x}=1$ and $\overline{y}=-1$. A word $a_1 \cdots a_k$ is \textit{Catalan} whenever $\overline{a}_1+\cdots+\overline{a}_i \geq 0$ for $1 \leq i \leq k-1$ and $\overline{a}_1+\cdots+\overline{a}_k=0$. The length of a Catalan word is always even. For $n \in \mN$, let $\Cat_n$ denote the set of all Catalan words of length $2n$. 
\end{definition}

\begin{example}\label{ex:Catn}\rm
We describe $\Cat_n$ for $0 \leq n \leq 3$. 
\[
\Cat_0=\{\m1\}, \hspace{4em} \Cat_1=\{xy\}, \hspace{4em} \Cat_2=\{xyxy,xxyy\},
\]
\[
\Cat_3=\{xyxyxy,xxyyxy,xyxxyy,xxyxyy,xxxyyy\}.
\]
\end{example}

\noindent We observe that for $0 \leq n \leq 3$ each $D_n$ is a linear combination of Catalan words of length $2n$. We now show that this observation is true for all $n \in \mN$. 

\begin{proposition}\label{prop:Dndecomp}\rm
For $n \in \mN$, $D_n$ is contained in the span of $\Cat_n$. 
\end{proposition}
\begin{proof}
For $n \in \mN$, by Definition \ref{def:alternating} we have that $\tilde{G}_n=xyxy \cdots xy$ where the $xy$ is repeated $n$ times. The word $\tilde{G}_n$ is Catalan by Definition \ref{def:Cat}. Note that the $q$-shuffle product of two Catalan words is a linear combination of Catalan words. The result follows by \eqref{convolution} and induction on $n$. 
\end{proof}

\begin{definition}\label{def:Dn}\rm
For $n \in \mN$ and a word $w \in \Cat_n$, let $(-1)^nD(w)$ denote the coefficient of $w$ in $D_n$. In other words, 
\begin{equation}\label{eq:Dn}
D_n=(-1)^n\sum_{w \in \Cat_n}D(w)w.
\end{equation}
\end{definition}

\begin{example}\rm
In the table below, we list the Catalan words $w$ of length $\leq 6$ and the corresponding $D(w)$. 
\begin{center}
\begin{tabular}{ c|ccccccccc } 
$w$ & $\m1$ & $xy$ & $xyxy$ & $xxyy$ & $xyxyxy$ & $xxyyxy$ & $xyxxyy$ & $xxyxyy$ & $xxxyyy$\\
\hline
$D(w)$ & 1 & 1 & 1 & $[2]_q^2$ & 1 & $[2]_q^2$ & $[2]_q^2$ & $[2]_q^4$ & $[2]_q^2[3]_q^2$
\end{tabular}
\end{center}
\end{example}

\noindent By \eqref{eq:Dn}, our goal of finding a closed formula for $D_n$ reduces to finding a closed formula for $D(w)$ where $w$ is Catalan. The following is the main theorem of this paper. 

\begin{theorem}\label{thm:main}\rm
For $n \in \mN$ and a word $w=a_1 \cdots a_{2n} \in \Cat_n$, we have 
\begin{equation}\label{eq:Dw}
D(w)=\prod_{i=1}^{2n}\left[\overline{a}_1+\cdots+\overline{a}_{i-1}+(\overline{a}_i+1)/2\right]_q.
\end{equation}
Moreover, 
\begin{equation}\label{eq:Dw1}
D(w)=E(w)^2,
\end{equation}
where
\begin{equation}\label{eq:Dw2}
E(w)=\prod_{\substack{1 \leq i \leq 2n \\ a_i=x}}[\overline{a}_1+\cdots+\overline{a}_i]_q=\prod_{\substack{1 \leq i \leq 2n \\ a_i=y}}[\overline{a}_1+\cdots+\overline{a}_{i-1}]_q.
\end{equation}
\end{theorem}

\begin{remark}\rm
There is a striking resemblance between \eqref{eq:Dw} and \cite[Definition 2.5]{ter_catalan}. While not explicitly used in our proofs, this resemblance did motivate our proof techniques and our interest in this entire topic. 
\end{remark}


\section{The proof of Theorem \ref{thm:main}}
In this section, we prove Theorem \ref{thm:main}. 

\begin{definition}\label{def:cDw}\rm
For $n \in \mN$ and a word $w=a_1 \cdots a_{2n} \in \Cat_n$, we define 
\begin{equation*}\label{Dw}
\cD(w)=\prod_{i=1}^{2n}\left[\overline{a}_1+\cdots+\overline{a}_{i-1}+(\overline{a}_i+1)/2\right]_q,
\end{equation*}
\[\cD_x(w)=\prod_{\substack{1 \leq i \leq 2n \\ a_i=x}}[\overline{a}_1+\cdots+\overline{a}_i]_q,\]
\[\cD_y(w)=\prod_{\substack{1 \leq i \leq 2n \\ a_i=y}}[\overline{a}_1+\cdots+\overline{a}_{i-1}]_q.\]
\end{definition}

\noindent In order to prove Theorem \ref{thm:main}, we establish the following for all Catalan words $w$: 
\begin{enumerate}[(i)]
\item $\cD(w)=D(w)$; 
\item $\cD(w)=\cD_x(w)\cD_y(w)$; 
\item $\cD_x(w)=\cD_y(w)$. 
\end{enumerate}

\noindent Item (i) will be achieved in Theorem \ref{thm:DD}. 

\medskip
\noindent Item (ii) will be achieved in Lemma \ref{rem:Dw}. 

\medskip
\noindent Item (iii) will be achieved in Lemma \ref{rem:Dwprofile0}. 

\begin{lemma}\label{rem:Dw}\rm
For any Catalan word $w$, we have 
\[
\cD(w)=\cD_x(w)\cD_y(w). 
\]
\end{lemma}
\begin{proof}
Note that $(\overline{x}+1)/2=1$ and $(\overline{y}+1)/2=0$, so the result follows by Definition \ref{def:cDw}. 
\end{proof}


\noindent Next we will show item (iii). In order to do this, we now recall the concept of elevation sequences and profiles. 

\begin{definition}\label{def:elevation_sequence}\rm
(See \cite[Definition 2.6]{ter_catalan}.) For $n \in \mN$ and a word $w=a_1 \cdots a_{n}$, its \textit{elevation sequence} is $(e_0,\ldots,e_{n})$, where $e_i=\overline{a}_0+\cdots+\overline{a}_i$ for $0 \leq i \leq n$. 
\end{definition}

\begin{example}\rm
In the table below, we list the Catalan words $w$ of length $\leq 6$ and the corresponding elevation sequences. 
\begin{center}
\begin{tabular}{ c|c } 
$w$ & elevation sequence of $w$\\
\hline
$\m1$ & $(0)$\\
$xy$ & $(0,1,0)$\\
$xyxy$ & $(0,1,0,1,0)$\\
$xxyy$ & $(0,1,2,1,0)$\\
$xyxyxy$ & $(0,1,0,1,0,1,0)$\\
$xxyyxy$ & $(0,1,2,1,0,1,0)$\\
$xyxxyy$ & $(0,1,0,1,2,1,0)$\\
$xxyxyy$ & $(0,1,2,1,2,1,0)$\\
$xxxyyy$ & $(0,1,2,3,2,1,0)$
\end{tabular}
\end{center}
\end{example}

\begin{definition}\label{def:profile}\rm
(See \cite[Definition 2.8]{ter_catalan}.) For $n \in \mN$ and a word $w=a_1 \cdots a_{n}$, its \textit{profile} is the subsequence of its elevation sequence consisting of the $e_i$ that satisfy one of the following conditions: 
\begin{itemize}
\item $i=0$; 
\item $i=n$; 
\item $1 \leq i \leq n-1$ and $e_{i+1}-e_i \neq e_i-e_{i-1}$. 
\end{itemize}
In other words, the profile of a word $w$ is the subsequence of the elevation sequence of $w$ consisting of the end points and turning points. 

\medskip
\noindent By a \textit{Catalan profile} we mean the profile of a Catalan word. 
\end{definition}

\begin{example}\rm
In the table below, we list the Catalan words $w$ of length $\leq 6$ and the corresponding profiles. 
\begin{center}
\begin{tabular}{ c|c } 
$w$ & profile of $w$\\
\hline
$\m1$ & $(0)$\\
$xy$ & $(0,1,0)$\\
$xyxy$ & $(0,1,0,1,0)$\\
$xxyy$ & $(0,2,0)$\\
$xyxyxy$ & $(0,1,0,1,0,1,0)$\\
$xxyyxy$ & $(0,2,0,1,0)$\\
$xyxxyy$ & $(0,1,0,2,0)$\\
$xxyxyy$ & $(0,2,1,2,0)$\\
$xxxyyy$ & $(0,3,0)$\\
\end{tabular}
\end{center}
\end{example}

\begin{lemma}\label{lem:Dwprofile0}\rm
For a Catalan word $w$ with profile $(l_0,h_1,l_1,\ldots,h_r,l_r)$, we have
\[ 
\cD_x(w)=\dfrac{[h_1]_q^!\cdots[h_r]_q^!}{[l_0]_q^!\cdots[l_r]_q^!},
\]
\[
\cD_y(w)=\dfrac{[h_1]_q^!\cdots[h_r]_q^!}{[l_0]_q^!\cdots[l_r]_q^!}.
\]
\end{lemma}
\begin{proof}
Follows from \cite[Lemma 2.10]{ter_catalan} by direct computation. 
\end{proof}

\begin{lemma}\label{rem:Dwprofile0}\rm
For any Catalan word $w$, we have 
\[\cD_x(w)=\cD_y(w).\]
\end{lemma}
\begin{proof}
Follows from Lemma \ref{lem:Dwprofile0}. 
\end{proof}

\begin{lemma}\label{lem:Dwprofile}\rm
For $n \in \mN$ and a word $w \in \Cat_n$ with profile $(l_0,h_1,l_1,\ldots,h_r,l_r)$, we have
\[\cD(w)=\cD_x(w)^2=\cD_y(w)^2=\left(\dfrac{[h_1]_q^!\cdots[h_r]_q^!}{[l_0]_q^!\cdots[l_r]_q^!}\right)^2.\]
\end{lemma}
\begin{proof}
Follows from Lemmas \ref{rem:Dw}, \ref{lem:Dwprofile0}, \ref{rem:Dwprofile0}. 
\end{proof}

\noindent Motivated by Lemma \ref{lem:Dwprofile}, we make the following definition. 

\begin{definition}\label{def:cDwprofile}\rm
Given a Catalan profile $(l_0,h_1,l_1,\ldots,h_r,l_r)$, define 
\[\cD(l_0,h_1,l_1,\ldots,h_r,l_r)=\left(\dfrac{[h_1]_q^!\cdots[h_r]_q^!}{[l_0]_q^!\cdots[l_r]_q^!}\right)^2.\]
\end{definition}

\begin{definition}\label{def:cDn}\rm
For $n \in \mN$, we define
\[\cD_n=(-1)^n\sum_{w \in \Cat_n}\cD(w)w.\]
We interpret $\cD_0=\m1$. 
\end{definition}

\noindent Next we will achieve a recurrence relation involving the $\cD_n$. This will be accomplished in Proposition \ref{thm:Dnrecursion}. 

\begin{lemma}\label{lem:Dprofile}\rm
For a Catalan profile $(l_0,h_1,l_1,\ldots,h_r,l_r)$ with $r \geq 1$, 
\begin{align*}
&\cD(l_0,h_1,l_1,\ldots,h_r,l_r)\\
&=\sum_{j=\xi}^{r-1}\cD(l_0,h_1,l_1,\ldots,h_j,l_j,h_{j+1}-1,l_{j+1}-1,\ldots,l_{r-1}-1,h_r-1,l_r)\left([h_{j+1}]_q^2-[l_j]_q^2\right),
\end{align*}
where $\xi=\max\{j \mid 0 \leq j \leq r-1,l_j=0\}$. 
\end{lemma}
\begin{proof}
To prove the above equation, consider the quotient of the right-hand side divided by the left-hand side. We will show that this quotient is equal to $1$. 

\medskip
\noindent By Definition \ref{def:cDwprofile}, the above quotient is equal to 
\begin{align*}
&\sum_{j=\xi}^{r-1}\frac{[l_{j+1}]_q^2\cdots[l_{r-1}]_q^2}{[h_{j+1}]_q^2\cdots[h_r]_q^2}\left([h_{j+1}]_q^2-[l_j]_q^2\right) \\
&=\frac{1}{[h_{\xi+1}]_q^2\cdots[h_r]_q^2}\sum_{j=\xi}^{r-1}[h_{\xi+1}]_q^2\cdots[h_j]_q^2[l_{j+1}]_q^2\cdots[l_{r-1}]_q^2\left([h_{j+1}]_q^2-[l_j]_q^2\right) \\
&=\frac{1}{[h_{\xi+1}]_q^2\cdots[h_r]_q^2}\sum_{j=\xi}^{r-1}\left([h_{\xi+1}]_q^2\cdots[h_{j+1}]_q^2[l_{j+1}]_q^2\cdots[l_{r-1}]_q^2-[h_{\xi+1}]_q^2\cdots[h_j]_q^2[l_j]_q^2\cdots[l_{r-1}]_q^2\right) \\
&=\frac{1}{[h_{\xi+1}]_q^2\cdots[h_r]_q^2}\left([h_{\xi+1}]_q^2\cdots[h_r]_q^2-[l_\xi]_q^2\cdots[l_{r-1}]_q^2\right) \\
&=1,
\end{align*}
where the last step follows from $l_\xi=0$. 
\end{proof}

\begin{lemma}\label{lem:innerprod0}\rm
For any Catalan word $w=a_1 \cdots a_m$, we have 
\[
\frac{q x \star w-q^{-1} w \star x}{q-q^{-1}}=\sum_{i=0}^ma_1 \cdots a_ixa_{i+1} \cdots a_m[1+2\overline{a}_1+\cdots+2\overline{a}_i]_q.
\] 
\end{lemma}
\begin{proof}
By the definition of the $q$-shuffle product, we have 
\begin{align*}
 & \frac{q x \star w-q^{-1} w \star x}{q-q^{-1}} \\
 & =\sum_{i=0}^m a_1 \cdots a_ixa_{i+1} \cdots a_m \hspace{0.25em} \frac{q^{1+2\overline{a}_1+\cdots+2\overline{a}_i}-q^{-1+2\overline{a}_{i+1}+\cdots+2\overline{a}_m}}{q-q^{-1}} \\
 & =\sum_{i=0}^m a_1 \cdots a_ixa_{i+1} \cdots a_m \hspace{0.25em} \frac{q^{1+2\overline{a}_1+\cdots+2\overline{a}_i}-q^{-1-2\overline{a}_1-\cdots-2\overline{a}_i}}{q-q^{-1}} \\
 & =\sum_{i=0}^m a_1 \cdots a_ixa_{i+1} \cdots a_m[1+2\overline{a}_1+\cdots+2\overline{a}_i]_q.
\end{align*}
\end{proof}

\noindent For notation convenience, we bring in a bilinear form on $\mV$. 

\begin{definition}\label{def:biform}\rm
(See \cite[Page 6]{ter_catalan}.) Let $(~,~):\mV \times \mV \to \mF$ denote the bilinear form such that $(w,w)=1$ for any word $w$ in $\mV$ and $(w,v)=0$ for any distinct words $w,v$ in $\mV$. 
\end{definition}
\noindent Observe that $(~,~)$ is non-degenerate and symmetric. For any word $w$ in $\mV$ and any $u \in \mV$, the scalar $(w,u)$ is the coefficient of $w$ in $u$. 

\begin{lemma}\label{lem:innerprod}\rm
For any word $v$ and any Catalan word $w=a_1 \cdots a_m$, consider the scalar 
\begin{equation}\label{eq:innerprod}
\left(\frac{(q x \star w-q^{-1} w \star x)y}{q-q^{-1}},v\right).
\end{equation}
\begin{enumerate}
\item If $v$ is Catalan and of length $m+2$, then the scalar \eqref{eq:innerprod} is equal to 
	\[\sum_i[1+2\overline{a}_1+\cdots+2\overline{a}_i]_q,\] 
	where the sum is over all $i$ $(0 \leq i \leq m)$ such that $v=a_1 \cdots a_ixa_{i+1} \cdots a_my$. 
\item If $v$ is not Catalan or is not of length $m+2$, then the scalar\eqref{eq:innerprod} is equal to $0$. 
\end{enumerate}
\end{lemma}
\begin{proof}
Follows from Lemma \ref{lem:innerprod0}. 
\end{proof}

\begin{lemma}\label{lem:Dvdecomp}\rm
For $n \geq 1$ and a word $v \in \Cat_n$, we have 
\[\cD(v)=\sum_{w \in \Cat_{n-1}}\cD(w)\left(\frac{(q x \star w-q^{-1} w \star x)y}{q-q^{-1}},v\right).\]
\end{lemma}
\begin{proof}
By Lemma \ref{lem:innerprod}, it suffices to show that $\cD(v)$ is equal to 
\begin{equation}\label{eq:cD(v)}
\sum_{w,i}\cD(w)[1+2\overline{a}_1+\cdots+2\overline{a}_i]_q,
\end{equation}
where the sum is over all ordered pairs $(w,i)$ such that $w=a_1 \cdots a_{2n-2} \in \Cat_{n-1}$ and $v=a_1 \cdots a_ixa_{i+1} \cdots a_{2n-2}y$. 

\medskip
\noindent Let $(l_0,h_1,l_1,\ldots,h_r,l_r)$ denote the profile of $v$ and let $\xi=\max\{j \mid 0 \leq j \leq r-1,l_j=0\}$. 

\medskip
\noindent To compute the sum \eqref{eq:cD(v)}, we study what kind of words $w$ are being summed over and what is the coefficient for each corresponding $\cD(w)$. 

\medskip
\noindent For any $w$ being summed over in \eqref{eq:cD(v)}, its profile must be of the form 
\[(l_0,h_1,l_1,\ldots,h_j,l_j,h_{j+1}-1,l_{j+1}-1,\ldots,l_{r-1}-1,h_r-1,l_r)\]
for some $j$ such that $\xi \leq j \leq r-1$. (If $j<\xi$, then the profile of $w$ contains $l_\xi-1=-1$, which means $w$ is not Catalan.) 

\medskip
\noindent For such $w$, the coefficient of $\cD(w)$ in \eqref{eq:cD(v)} is 
\[\sum_{s=l_j}^{h_{j+1}-1}[1+2s]_q,\]
which is equal to 
\[[h_{j+1}]_q^2-[l_j]_q^2\]
by direct computation. 

\medskip
\noindent Therefore, by Lemma \ref{lem:Dprofile} we have 
\begin{align*}
&\sum_{w,i}\cD(w)[1+2\overline{a}_1+\cdots+2\overline{a}_i]_q \\
&=\sum_{j=\xi}^{r-1}\cD(l_0,h_1,l_1,\ldots,h_j,l_j,h_{j+1}-1,l_{j+1}-1,\ldots,l_{r-1}-1,h_r-1,l_r)\left([h_{j+1}]_q^2-[l_j]_q^2\right) \\
&=\cD(l_0,h_1,l_1,\ldots,h_r,l_r) \\
&=\cD(v). 
\end{align*}
\end{proof}

\begin{proposition}\label{thm:Dnrecursion}\rm
For $n \geq 1$, 
\begin{equation}\label{eq:Dnrecursion}
\cD_n=\frac{(q^{-1} \cD_{n-1} \star x-q x \star \cD_{n-1})y}{q-q^{-1}}.
\end{equation}
\end{proposition}
\begin{proof}
Given any word $v$, we will show that its inner product with the right-hand side of \eqref{eq:Dnrecursion} coincides with $(\cD_n,v)$. 

\medskip
\noindent If $v$ does not have length $2n$, then the two inner products are both $0$. 

\medskip
\noindent If $v$ is not Catalan, then $(\cD_n,v)=0$ by Definition \ref{def:cDn}, and 
\[\left(\frac{(q^{-1} \cD_{n-1} \star x-q x \star \cD_{n-1})y}{q-q^{-1}},v\right)=0\]
by Definition \ref{def:cDn} and Lemma \ref{lem:innerprod}. 

\medskip
\noindent If $v \in \Cat_n$, then by Definition \ref{def:cDn} and Lemma \ref{lem:Dvdecomp},  
\begin{align*}
&\left(\frac{(q^{-1} \cD_{n-1} \star x-q x \star \cD_{n-1})y}{q-q^{-1}},v\right) \\
&=(-1)^n\sum_{w \in \Cat_{n-1}}\cD(w)\left(\frac{(q x \star w-q^{-1} w \star x)y}{q-q^{-1}},v\right) \\
&=(-1)^n\cD(v) \\
&=(\cD_n,v). 
\end{align*}
\end{proof}

\begin{definition}\label{def:cDt}\rm
(See \cite[Definition 9.11]{ter_alternating}.) We define a generating function
\[
\cD(t)=\sum_{n \in \mN} \cD_nt^n, 
\]
where $\cD_n$ is from Definition \ref{def:cDn}. 
\end{definition}

\noindent Next we will show that $\cD(t)=D(t)$. To do this, we will show that $\cD(t)$ is the multiplicative inverse of $\tilde{G}(t)$ with respect to $\star$. This will be accomplished in Proposition \ref{prop:inverse}. 

\begin{lemma}\label{lem:inverse0}\rm
For $k \in \mN$, we have 
\[
q\tilde{G}_k \star x=(q-q^{-1})W_{-k}+q^{-1}x \star \tilde{G}_k. 
\]
\end{lemma}
\begin{proof}
Follows from the definition of $\star$ by direct computation. 
\end{proof}

\begin{lemma}\label{lem:inverse}\rm
For $n \geq 1$, 
\begin{equation}\label{convolution2}
\cD_n=-\sum_{k=1}^n\tilde{G}_k \star \cD_{n-k}.
\end{equation}
\end{lemma}
\begin{proof}
We use induction on $n$. 

\medskip
\noindent First assume that $n=1$. Then \eqref{convolution2} holds because 
\[\cD_0=\m1, \hspace{4em} \cD_1=-xy, \hspace{4em} \tilde{G}_1=xy.\]
Next assume that $n \geq 2$. By induction, 
\begin{equation}\label{i1}
\cD_{n-1}=-\sum_{k=1}^{n-1}\tilde{G}_k \star \cD_{n-1-k}.
\end{equation}
In order to prove \eqref{convolution2}, it suffices to show 
\begin{equation}\label{i1.5}
\sum_{k=1}^{n-1}\tilde{G}_k \star \cD_{n-k}=-\cD_n-\tilde{G}_n. 
\end{equation}
For $1 \leq k \leq n-1$ we examine the $k$-summand in \eqref{i1.5}. We use the following notation: for a word $w$ ending with the letter $y$, the word $wy^{-1}$ is obtained from $w$ by removing the rightmost $y$. Furthermore, for a linear combination $A$ of words ending in $y$, the element $Ay^{-1}$ is obtained from $A$ by removing the rightmost $y$ of each word in the linear combination. 

\medskip
\noindent Note that $\tilde{G}_k$ is a word ending in $y$, and $\cD_{n-k}$ is a linear combination of Catalan words which end in $y$ by Definition \ref{def:Cat}, so 
\begin{equation}\label{i2}
\tilde{G}_k \star \cD_{n-k}=(\tilde{G}_ky^{-1} \star \cD_{n-k})y+(\tilde{G}_k \star \cD_{n-k}y^{-1})y. 
\end{equation}
We focus on the second term of the right-hand side of \eqref{i2}. By Proposition \ref{thm:Dnrecursion} and Lemma \ref{lem:inverse0}, we have
\begin{equation*}\label{i3}
\begin{split}
&\tilde{G}_k \star \cD_{n-k}y^{-1} \\
&=-\frac{1}{q-q^{-1}}\tilde{G}_k \star (qx \star \cD_{n-k-1}-q^{-1}\cD_{n-k-1} \star x) \\
&=-\frac{q}{q-q^{-1}}\tilde{G}_k \star x \star \cD_{n-k-1}+\frac{q^{-1}}{q-q^{-1}}\tilde{G}_k \star \cD_{n-k-1} \star x \\
&=-W_{-k} \star \cD_{n-k-1}-\frac{q^{-1}}{q-q^{-1}}x \star \tilde{G}_k \star \cD_{n-k-1}+\frac{q^{-1}}{q-q^{-1}}\tilde{G}_k \star \cD_{n-k-1} \star x.
\end{split}
\end{equation*}
By the above comment, and since $\tilde{G}_ky^{-1}=W_{-k+1}$, we can write \eqref{i2} as 
\begin{equation*}\label{i4}
\begin{split}
&\tilde{G}_k \star \cD_{n-k} \\
&=(W_{-k+1} \star \cD_{n-k})y-(W_{-k} \star \cD_{n-k-1})y \\
&\hspace{4em} -\frac{q^{-1}}{q-q^{-1}}(x \star \tilde{G}_k \star \cD_{n-k-1})y+\frac{q^{-1}}{q-q^{-1}}(\tilde{G}_k \star \cD_{n-k-1} \star x)y.
\end{split}
\end{equation*}
We now sum the above equation over $k$ from $1$ to $n-1$, using \eqref{i1} and Proposition \ref{thm:Dnrecursion}. We have
\begin{equation*}
\begin{split}
&\sum_{k=1}^{n-1}\tilde{G}_k \star \cD_{n-k} \\
&=(W_0 \star \cD_{n-1})y-(W_{-n+1} \star \cD_0)y+\frac{q^{-1}}{q-q^{-1}}(x \star \cD_{n-1})y-\frac{q^{-1}}{q-q^{-1}}(\cD_{n-1} \star x)y \\
&=(x \star \cD_{n-1})y-\tilde{G}_n+\frac{q^{-1}}{q-q^{-1}}(x \star \cD_{n-1})y-\frac{q^{-1}}{q-q^{-1}}(\cD_{n-1} \star x)y \\ 
&=\frac{q}{q-q^{-1}}(x \star \cD_{n-1})y-\frac{q^{-1}}{q-q^{-1}}(\cD_{n-1} \star x)y-\tilde{G}_n \\
&=-\cD_n-\tilde{G}_n.
\end{split}
\end{equation*}
We have verified \eqref{i1.5}, and \eqref{convolution2} follows. 
\end{proof}

\begin{definition}\label{def:zeta}\rm
(See \cite[Page 5]{ter_catalan}.) Let $\zeta:\mV \to \mV$ denote the $\mF$-linear map such that 
\begin{itemize}
\item $\zeta(x)=y$, 
\item $\zeta(y)=x$, 
\item For any word $a_1 \cdots a_m$, 
	\[\zeta(a_1 \cdots a_m)=\zeta(a_m) \cdots \zeta(a_1).\]
\end{itemize}
\end{definition}

\noindent By the above definition, $\zeta$ is an antiautomorphism on the free algebra $\mV$. One can routinely check using the definition of $\star$ that $\zeta$ is also an antiautomorphism on the $q$-shuffle algebra $\mV$. Moreover, $\zeta$ fixes $\tilde{G}_n$ and $\cD_n$ for all $n \in \mN$. 

\begin{proposition}\label{prop:inverse}\rm
We have 
\[
\tilde{G}(t) \star \cD(t)=\m1=\cD(t) \star \tilde{G}(t). 
\]
\end{proposition}
\begin{proof}
We have $\tilde{G}_0=\m1$ and $\cD_0=\m1$. By Lemma \ref{lem:inverse}, for any $n \geq 1$ we have 
\begin{equation*}
\sum_{k=0}^n\tilde{G}_k \star \cD_{n-k}=0.
\end{equation*}
By these comments,  
\begin{equation}\label{i6}
\tilde{G}(t) \star \cD(t)=\m1.
\end{equation}
Applying $\zeta$ to \eqref{i6}, we have 
\[\cD(t) \star \tilde{G}(t)=\m1.\]
\end{proof}

\begin{theorem}\label{thm:DD}\rm
The following hold. 
\begin{enumerate}[(i)]
\item $\cD(t)=D(t)$. 
\item $\cD_n=D_n$ for any $n \in \mN$.
\item $\cD(w)=D(w)$ for any Catalan word $w$. 
\end{enumerate}
\end{theorem}
\begin{proof}
Comparing Lemma \ref{lem:Dtinverse} and Proposition \ref{prop:inverse}, we obtain item (i). Item (ii) follows from item (i) by Definitions \ref{def:Dt} and \ref{def:cDt}. Item (iii) follows from item (ii) by Definitions \ref{def:Dn} and \ref{def:cDn}. 
\end{proof}

\noindent This finishes our proof of Theorem \ref{thm:main}. 

\section{Some facts about $\{D_n\}_{n \in \mN}$}
In this section, we state some facts about $\{D_n\}_{n \in \mN}$ that we find attractive. 

\begin{proposition}\label{prop:Dnpoly}\rm
(See \cite[Lemma 9.7]{ter_alternating}.) For $n \geq 1$, 
\begin{itemize}
\item $D_n$ is a polynomial in $\tilde{G}_1,\ldots,\tilde{G}_n$ of degree $n$, where each $\tilde{G}_i$ is given the degree $i$, 
\item $\tilde{G}_n$ is a polynomial in $D_1,\ldots,D_n$ of degree $n$, where each $D_i$ is given the degree $i$. 
\end{itemize}
\end{proposition}

\begin{proposition}\label{prop:Dncomm}\rm
(See \cite[Lemma 9.10]{ter_alternating}.) For $n,m \in \mN$, 
\[
D_n \star \tilde{G}_m=\tilde{G}_m \star D_n, \hspace{4em} D_n \star D_m=D_m \star D_n. 
\]
\end{proposition}


\begin{proposition}\label{cor:Dnrecursion1}\rm
For $n \geq 1$, 
\begin{equation}\label{eq:Dnrecursion1}
D_n=\frac{(q^{-1} D_{n-1} \star x-q x \star D_{n-1})y}{q-q^{-1}}.
\end{equation}
\end{proposition}
\begin{proof}
Follows from Proposition \ref{thm:Dnrecursion} and Theorem \ref{thm:DD}. 
\end{proof}

\begin{proposition}\label{cor:Dnrecursion2}\rm
For $n \geq 1$, 
\[D_n=\frac{x(q^{-1}y \star D_{n-1}-qD_{n-1} \star y)}{q-q^{-1}}.\]
\end{proposition}
\begin{proof}
Apply the antiautomorphism $\zeta$ to each side of \eqref{eq:Dnrecursion1}, and note that $D_n$ is invariant under $\zeta$. 
\end{proof}

\noindent Recall that for a linear combination $A$ of words ending in $y$, the element $Ay^{-1}$ is obtained from $A$ by removing the rightmost $y$ of each word. We make a similar notation that for a linear combination $B$ of words starting with $x$, the element $x^{-1}B$ is obtained from $B$ by removing the leftmost $x$ of each word. 
\begin{proposition}\label{cor:Dnrecursion3}\rm
For $n \geq 2$, 
\begin{equation}\label{eq:Dnrecursion3}
x^{-1}D_ny^{-1}+D_{n-1}=\frac{q^{-1}x^{-1}D_{n-1} \star x-q^3x \star x^{-1}D_{n-1}}{q-q^{-1}}.
\end{equation}
\end{proposition}
\begin{proof}
By the definition of the $q$-shuffle product, we have 
\begin{equation*}\label{rec1}
x \star D_{n-1}=xD_{n-1}+q^2x(x \star x^{-1}D_{n-1}), 
\end{equation*}
\begin{equation*}\label{rec2}
D_{n-1} \star x=xD_{n-1}+x(x^{-1}D_{n-1} \star x). 
\end{equation*}
\medskip
The result follows from Proposition \ref{cor:Dnrecursion1} and the two equations above. 
\end{proof}

\begin{proposition}\label{cor:Dnrecursion4}\rm
For $n \geq 2$, 
\[x^{-1}D_ny^{-1}+D_{n-1}=\frac{q^{-1}y \star D_{n-1}y^{-1}-q^3D_{n-1}y^{-1} \star y}{q-q^{-1}}.\]
\end{proposition}
\begin{proof}
Apply the antiautomorphism $\zeta$ to each side of \eqref{eq:Dnrecursion3}, and note that $D_n$ is invariant under $\zeta$. 
\end{proof}

\noindent Next we mention some PBW bases for $U_q^+$ that involve $\{D_{n+1}\}_{n \in \mN}$. The readers may refer to \cite[Definition 2.1]{ter_alternating} for a formal definition of a PBW basis. 

\begin{proposition}\label{prop:DnPBW1}\rm
The elements $\{W_{-n}\}_{n \in \mN}$, $\{D_{n+1}\}_{n \in \mN}$, $\{W_{n+1}\}_{n \in \mN}$ form a PBW basis for $U_q^+$ in any linear order that satisfies one of the following: 
\begin{enumerate}
\item $W_{-i}<D_{j+1}<W_{k+1}$ for $i,j,k \in \mN$; 
\item $W_{k+1}<D_{j+1}<W_{-i}$ for $i,j,k \in \mN$; 
\item $W_{k+1}<W_{-i}<D_{j+1}$ for $i,j,k \in \mN$; 
\item $W_{-i}<W_{k+1}<D_{j+1}$ for $i,j,k \in \mN$; 
\item $D_{j+1}<W_{k+1}<W_{-i}$ for $i,j,k \in \mN$; 
\item $D_{j+1}<W_{-i}<W_{k+1}$ for $i,j,k \in \mN$. 
\end{enumerate}
\end{proposition}
\begin{proof}
Follows from \cite[Theorem 10.1]{ter_alternating} and Propositions \ref{prop:Dnpoly}, \ref{prop:Dncomm}. 
\end{proof}

\begin{proposition}\label{prop:DnPBW2}\rm
The elements $\{E_{n\delta+\alpha_0}\}_{n \in \mN}$, $\{D_{n+1}\}_{n \in \mN}$, $\{E_{n\delta+\alpha_1}\}_{n \in \mN}$ form a PBW basis for $U_q^+$ in the following linear order: 
\[E_{\alpha_0}<E_{\delta+\alpha_0}<E_{2\delta+\alpha_0}<\cdots<D_1<D_2<D_3<\cdots<E_{2\delta+\alpha_1}<E_{\delta+\alpha_1}<E_{\alpha_1}.\]
\end{proposition}
\begin{proof}
Follows from \cite[Section 5]{damiani}, \cite[Theorem 1.7]{ter_catalan}, \cite[Proposition 11.9]{ter_alternating}, and Proposition \ref{prop:Dncomm}. 
\end{proof}

\section{Acknowledgments}
The author is currently a Math Ph.D. student at the University of Wisconsin-Madison. The author would like to thank his supervisor, Professor Paul Terwilliger, for suggesting the paper topic and giving many helpful comments. The author would like to thank Pascal Baseilhac for comments on this paper. The author would also like to thank his high school Math teacher, Yuefeng Feng, for guiding the author through an early tour of the fascinating world of combinatorics.

\bigskip
\noindent Chenwei Ruan \\
Department of Mathematics \\
University of Wisconsin \\
480 Lincoln Drive \\
Madison, WI 53706-1388 USA \\
email: {\tt cruan4@wisc.edu }

\end{document}